\documentclass[preprint,10pt]{elsarticle}
\usepackage{amsmath,amsopn,amssymb,epsfig}

\newtheorem{theorem}{Theorem}[section]
\newtheorem{lemma}{Lemma}[section]

\newtheorem{example}{Example}[section]
\newtheorem{Definition}{Definition}[section]
\newtheorem{corollary}{Corollary}[section]
\newproof{proof}{Proof}
\newproof{pot}{Proof of Theorem \ref{thm2}}
\newcommand{\bbb}{\begin{bmatrix}}
\newcommand{\eb}{\end{bmatrix}}

\begin{document}

\date{}
\begin{frontmatter}
\title{
A note on Sylvester-type equations }
\author{Matthew M. Lin \fnref{fn1}}
\ead{mhlin@ccu.edu.tw}
\address{Department of Mathematics, National Chung Cheng University, Chia-Yi 621, Taiwan.}
\author{Chun-Yueh Chiang\fnref{fn2}\corref{cor1}}
\ead{chiang@nfu.edu.tw}
\address{Center for General Education, National Formosa
University, Huwei 632, Taiwan.}
\cortext[cor1]{Corresponding author}
\fntext[fn1]{
The first author was supported in part by the National Center for Theoretical Sciences
of Taiwan and by the Ministry of Science and Technology of Taiwan under grant NSC
101-2115-M-194-007-MY3.
}
\fntext[fn2]{The second author was supported
in part by the National Center for Theoretical Sciences
of Taiwan and by
 the Ministry of Science and Technology of Taiwan under grant
NSC102-2115-M-150-002.}
\date{ }

\begin{abstract}
This work is to provide a comprehensive treatment of the relationship between the theory of the generalized (palindromic) eigenvalue problem and the theory of the Sylvester-type equations. Under a regularity assumption for a specific matrix pencil, we show that the solution of the $\star$-Sylvester matrix equation is uniquely determined and can be obtained by considering  its corresponding deflating subspace. We also propose an iterative method with quadratic convergence to compute the stabilizing  solution of the $\star$-Sylvester matrix equation via the well-developed palindromic doubling algorithm. We believe that our discussion is the first which implements the tactic of the deflating subspace for solving Sylvester equations and could give rise to the possibility of developing an advanced and effective solver for different types of matrix equations.
\end{abstract}


\bigskip


\end{frontmatter}

\section{Introduction}
The discussion of matrix eigenvalue problems in terms of deflating subspaces and the solutions of matrix equations has received a great deal of attention, particularly in its wide range of important applications in control theory~\cite{Gohberg1982,Mehrmann1991,Zhou1996}. It is famous that the solutions of Riccati and  matrix polynomial equations can be obtained by computing its corresponding deflating subspaces and invariant subspaces, respectively. For example, consider a matrix (unilateral) polynomial equation with degree $m$
\begin{align}\label{MM}
A_m X^m+A_{m-1} X^{m-1}+\cdots+A_0=0,
\end{align}
where the complex coefficient matrices $A_0,\cdots,A_m$ and unknown matrix $X$ are of size $n\times n$. Let the \emph{companion} matrix pencil $\mathcal{M}-\lambda \mathcal{L}\in \mathbb{C}^{mn \times mn} \times \mathbb{C}^{mn \times mn}$ be defined as
\begin{align}\label{MLL}
\mathcal{M}=
\bbb
     0 & I & 0 & \cdots & 0 \\
     0 & 0 & I &  & 0 \\
     \vdots & \vdots & \ddots & 0 &  \\
     0 & 0 & \cdots & 0 & I \\
     -A_0 & -A_1 & \cdots & -A_{m-2} & -A_{m-1} \\
\eb,\,\,
\mathcal{L}=
\bbb
I & 0 &\cdots&\cdots& 0 \\
0 & I & & &0\\
\vdots & &\ddots& & \\
0 & 0 &  & & A_m \\
\eb,
\end{align}
where $I$ and $0$ denote respecitively the identity and zero matrix with appropriate
size. It is easy to show
\begin{align}\label{ML}
\mathcal{M} U =\mathcal{L} UX,
\end{align}
with the ``column matrix'' $U=\bbb I \\ X\\\vdots\\ X^{m-1}\eb$. According to the equality~\eqref{ML}, the information of solution $X$ of ~\eqref{MM} is embedded in the eigeninformation of the generalized eigenvalue problem~\eqref{MLL}.
Our main point in this work is to connect deflating subspaces, or more specifically invariant subspaces, with solutions of the following linear matrix equations,
\begin{subequations}
\begin{itemize}
\item[1.] Standard Sylvester matrix equation:
\begin{align}\label{sylvester}
AX+XB=C,\quad A\in\mathbb{C}^{m \times m},B\in\mathbb{C}^{n \times n},C,X\in\mathbb{C}^{m \times n},
\end{align}
\item[2.] $\star$-Sylvester matrix equation:
\begin{align}\label{CTS}
AX+X^\star B=C,\quad A,B,C,X\in\mathbb{C}^{n \times n},
\end{align}
where $\star=\top$ is the transport operator or  $\star=H$ is the Hermitian operator.
\end{itemize}
\end{subequations}

The study of the Sylvester equations of the form~\eqref{sylvester}
has been widely discussed in theory and applications. Especially, ~\eqref{sylvester} plays an indispensable role in a variety of fields of control theory \cite{Lancaster1995}. On the other hand, the study of the $\star$-Sylvester equations stems from the treatment of completely integrable mechanical systems.
When $\star=\top$, a formula for the solution of \eqref{CTS} in terms of generalized inverse and some necessary and sufficient conditions for the existence of the solution of \eqref{CTS} were proposed in \cite{Piao2007}. In recent years, an extensive 
amount of iterative methods based on the conjugate gradient method were
studied and developed for solving the  generalized  $T$-Sylvester equation
\[ A X B + C X^T D = E \ , \ \ \ A,B,C,D,E,X \in \mathbb{R}^{n \times n}. \]
See, e.g., \cite{Su2010, Hajariann2013} and the references cited therein.
However, for the matrix equation~\eqref{CTS} especially with $\star = H$, there are not many references in the literature and in particular, the developed method so far for solving~\eqref{CTS} is through the application of the generalized Schur form of the pencil $A-\lambda B^\star$~\cite{Chiang2012}. In this paper, the solutions of the standard and $\star$-Sylvester equations are studied in terms of the study of invariant subspace and deflating subspace methods. In particular, we are mainly interested in the square cases when $m = n$ for $\star$-Sylvester equations. Other relative works can be found
in \cite{Hajariann2013,De2011}.

The paper is organized as follows. In Section 2 we recall some properties of the generalized eigenvalue problems. In Section 3 we give an invariant subspace method for computing the solution of the standard Sylvester equation~\eqref{sylvester}. In Section 4 we show how the deflating subspace method can be applied to solve {$\star$-Sylvester matrix equation}~\eqref{CTS}.  In Section 5 a quadratic convergence method is provided for finding the stabilizing solution of ~\eqref{CTS} and concluding remarks are given in Section 6.

\section{Preliminaries} \label{sec_background}
In this section we briefly review some properties of matrix pencil which are required in the statements and the proofs in the following sections. To facilitate our discussion, we use $\sigma(A)$ and $\sigma(A-\lambda B)$ to denote the spectrum of the matrix $A$ and the matrix pair $A-\lambda B$, respectively and the notion $\sim$ to denote the spectral equivalent condition, that is, $A-\lambda B \sim \widetilde{A} -\lambda \widetilde{B}$ implies that $\sigma(A-\lambda B) = \sigma(\widetilde{A} -\lambda \widetilde{B})$.

Given a matrix pencil $A-\lambda B$, the pair $A-\lambda B$ is said to be \emph{regular} if $\det(A-\lambda B) \neq 0$ for some $\lambda\in \mathbb{C}$. One strategy to analyze the eigeninformation is to transform one matrix pencil to its simplified and equivalent form. That is, two matrix pencils $A-\lambda B$ and $\widetilde{A} -\lambda \widetilde{B}$ are said to be equivalent if and only if there exist two nonsingular matrices $P$ and $Q$ such that
 \begin{equation*}
 P(A-\lambda B) Q = \widetilde{A} -\lambda \widetilde{B}.
 \end{equation*}

Like similarity discussed in the ordinary eigenvalue problem, the property of equivalence preserves eigenvalues and transforms eigenvectors in a similar way. This result can be easily understood through the following well-known result given by
Weierstrass~\cite{Weierstrass1868} and Kronecter~\cite{Kronecker1890} and is also discussed in~\cite[Defintion 1.13]{Stewart1990}.
\begin{theorem}\label{Weierstrass}
[Weierstrass canonical form for a matrix pencil]
Let $A-\lambda B$ be a regular pair. Then there exist nonsingular matrices $P$ and $Q$ such that
\begin{equation}~\label{wJform}
PAQ = \left[\begin{array}{cc} J & 0 \\0 & I\end{array}\right]  \mbox{ and }
PBQ = \left[\begin{array}{cc} I & 0 \\0 & N\end{array}\right],
\end{equation}
where $J$ and $N$ are in Jordan canonical form and diagonal entries of $N$ are zero (i.e., $N$ is nilpotent). Also, we can simplify this result by
the notation of direct sum. That is, $A-\lambda B \sim (J\oplus I)-\lambda (I \oplus N)$. Sometimes the canonical form in~\eqref{wJform} is also called the \emph{Kronecker canonical form}.
\end{theorem}

From theorem~\ref{Weierstrass}, it is easy to see that if $\lambda$ is an eigenvalue of $A-\lambda B$ with eigenvector $\mathbf{x}$, then $\lambda$ is an eigenvalue of $PAQ - \lambda PBQ$ with eigenvector $Q^{-1}\mathbf{x}$.
Note that for an $n\times n$ matrix $A$, the generalized eigenvectors of $A$ span the entire $\mathbb{R}^n$ space. This property is also true for every regular matrix pencil and is demonstrated as follows. For a detailed proof, the reader is referred to~\cite[Theorem 7.3]{Gohberg1982}.

\begin{theorem}\label{lancaster}
Given a matrix pair of $n\times n$ matrix $A$ and $B$, if the matrix pencil $A-\lambda B$ is regular, then its Jordan chains corresponding to all finite and infinite eigenvalues carry the full spectral information about the matrix pencil
and consists of $n$ linearly independent vectors.
 \end{theorem}

\section{Standard Sylvester equations}
It is true that the standard Sylvester matrix equation~\eqref{sylvester} can be viewed as the application of Newton method to the nonsymmetric algebraic Riccati equation (NARE):
\begin{equation}\label{nare}
C+XA+DX-XBX = 0,
\end{equation}
where $X\in\mathbb{C}^{m\times n}$ is the unknown, and where the coefficients are $A\in\mathbb{C}^{n\times n}$, $B\in\mathbb{C}^{n\times m}$, $C\in\mathbb{C}^{m\times n}$, and $D\in\mathbb{C}^{m\times m}$,
and the solution $X$ of~\eqref{nare} can be solved by considering the invariant subspace of the Hamiltonian-like matrix~\cite{GuoXULin2006}
\begin{equation}
\mathcal{H} = \left[\begin{array}{cc}-A & B \\C & D\end{array}\right].
\end{equation}
But, unlike the NARE, the solvability of the standard Sylvester equation don't require rigorous constraints on the matrix $\mathcal{H}$ as given in~\cite{GuoXULin2006}. Instead, we want to show that the invariant subspace method for solving~\eqref{sylvester} can be developed by simply applying the following solvability condition of the Sylvester equation~\eqref{sylvester}~\cite{Horn1994}.
\begin{theorem}
The equation of~\eqref{sylvester} has a unique solution for each $C$ if and only if $\sigma(-A)\cap\sigma(B) = \phi$. 
\end{theorem}

In this section, we want to establish the idea of invariant subspace. Its theory is an application of the following result~\cite{Chu1987}.
\begin {theorem}\label{Eric0}
Given two  regular matrix pencils $A_i-\lambda B_i\in\mathbb{C}^{n_i\times n_i}$, $1\leq i \leq 2$. Consider the following equations with respect to $U,V\in\mathbb{C}^{n_1\times n_2}$
\begin{subequations}\label{CM}
\begin{align}
A_1 U &= V A_2,\\
B_1 U &= V B_2.
\end{align}
\end{subequations}
If $\sigma(A_1-\lambda B_1)\cap\sigma(A_2-\lambda B_2)=\phi$, then~\eqref{CM} has a unique solution $U=V=0$.
\end {theorem}

Note that Theorem~\ref{Eric0} yields a Corollary which is
simple, but useful in our subsequent discussion.
\begin{corollary}\label{Eric}
Let $A\in\mathbb{C}^{n\times n}$ and $T\in\mathbb{C}^{k\times k}$. If $\sigma(A)\cap\sigma(T)=\phi$, then the equation
\begin{align*}
AU=U T
\end{align*}
have the unique solution $U=0\in\mathbb{C}^{n\times k}$.
\end{corollary}

Now we have enough tools for discussing the invariant subspace corresponding to~\eqref{sylvester}.

\begin {theorem}\label{thm2}
Let $A$, $B$, $C$ be the matrices given in~\eqref{sylvester}.  If $\mathcal{M}=\bbb -A & C\\0 & B \eb \in \mathbb{C}^{(m+n)\times (m+n)}$
and if $\sigma(-A)\cap \sigma(B)=\phi$, let us write
 \begin{align*}
\mathcal{M}\bbb U \\ V \eb= \bbb U\\V \eb T,
\end{align*}
where $\bbb U\\V \eb$ is full rank. Then, we have
\begin{enumerate}
\item $V=0$ and $-A U=U T$ if $\sigma (T)=\sigma(-A)$.
\item $V$ is nonsingular and $X=UV^{-1}$ is the solution  of~\eqref{sylvester} if $\sigma (T)=\sigma(B)$.
\end{enumerate}
\end {theorem}

\begin{proof}
Since $\sigma(-A)\cap \sigma(B)=\phi$, it follows from Corollay~\ref{Eric} that the equation $BV = -VA$ has only one solution $V = 0$. This proves part 1.


Consider two equations
\begin{align*}
\mathcal{M}\bbb U_1 \\ 0 \eb &= \bbb U_1\\0 \eb
T_1,\,
\mbox{with }
\sigma(T_1)=\sigma(-A),\\
\mathcal{M}\bbb U_2 \\ V_2 \eb &= \bbb U_2\\V_2 \eb T_2,\,
 \mbox{with }
\sigma(T_2)=\sigma(B).
\end{align*}
From Theorem~\ref{lancaster}, all column vectors of $\bbb U_1\\0 \eb$ and $\bbb U_2\\V_2 \eb$ are linearly independent. This implies that $\bbb U_1 & U_2\\0 & V_2 \eb$ is nonsingular, that is, $V_2$ is nonsingular.
Observe further that
 \begin{subequations}
\begin{align*}
-A U_2 + C V_2 &= U_2 T_2,\\
B V_2 &= V_2T_2.
\end{align*}
\end{subequations}
Upon substitution, we see that
$A  (U_2V_2^{-1}) + (U_2V_2^{-1}) B = C$
 and the proof is complete.

\end{proof}

From Theorem~\ref{thm2}, we know that, in order to compute the solution of~\eqref{sylvester}, it is sufficient to compute a base for the invariant subspace associated with the eigenvalues of $B$. Some acceleration iterative methods like doubling algorithm \cite{GuoXULin2006} for finding the unique solution of ~\eqref{sylvester} are based on the relationship between ~\eqref{sylvester} and invariant subspace in Theorem~\ref{thm2}. We don't further discuss here.

\section{$\star$-Sylvester equations}
Before demonstrating the unique solvability conditions~\eqref{CTS}, we need to define that a subset $\Lambda  = \{\lambda_1,\ldots,\lambda_n\}$ of complex numbers is said to be \emph{$\star$-reciprocal free} if and only if $\lambda_i\neq1/\lambda_j^\star$ for $1\leq i,j \leq n$. This definition also regards $0$ and $\infty$ as reciprocals of each other.
Note that the necessary and sufficient conditions for unique solvability
of~\eqref{CTS} are given in~\cite{Wimmer1994} by means of Roth's criterion.  Consult also~\cite[Lemma 5.10]{Byers2006} and~\cite[Lemma 8]{Kressner2009}, where
solvability conditions
for the $\star$-Sylvester equations with $m = n$ were obtained, without
considering the details of the solution process. That is, we have the following solvability conditions of~\eqref{CTS}.
\begin{theorem}\label{solvable}
Suppose that the pencil $A-\lambda B^\star$ is regular, the $\star$-Sylvester matrix equation~\eqref{CTS} is uniquely solvable if and only if
\begin{itemize}
\item[1.]
 For $\star = \top$, $\sigma(A^\top-\lambda B)\backslash \{1\}$ is reciprocal free, and whenever $1\in\sigma(A^\top-\lambda B)$, $1$ is simple;
\item[2.]
  For $\star = H$, $\sigma(A^H-\lambda B)$ is reciprocal free, and  $|\lambda| \neq 1$,  whence $\lambda\in\sigma(A^H-\lambda B)$.
 \end{itemize}
\end{theorem}

 The most straightforward way to find the solution of~\eqref{CTS}  is through the analysis of its corresponding \emph{palindromic} eigenvalue problem \cite{Chu2010,NLA:NLA612,Mackey2006,doi:10.1137/050628350}
\begin{equation}\label{eq:PEP}
\mathcal{Q}(\lambda)x: = (\mathcal{Z}^\star-\lambda \mathcal{Z})x=0,
\end{equation}
where
\begin{equation}\label{eq:Zmat}
\mathcal{Z} =\bbb 0 & B \\ A & -C \eb\in\mathbb{C}^{2n\times 2n},
\end{equation}
and the discussion of the deflating subspace of~\eqref{eq:PEP}. Again, the operator $(\cdot)^\star$ denotes either the transpose ($\top$), or the conjugate transpose ($H$) of a matrix.  We interpret  both cases in a unified way hereafter. The name ``palindromic'' stems from the invariant property of $\mathcal{Q(\lambda)}$
under reversing the order and taking (conjugate) transpose of the coefficient matrices, that is,
\begin{equation*}
\mbox{rev} \mathcal{Q}(\lambda) = \mathcal{Q}^\star(\lambda),\,
\mbox{where } \mbox{rev} \mathcal{Q}(\lambda) : =
(\mathcal{Z}-\lambda \mathcal{Z}^\star).
\end{equation*}

The corresponding eigenvalue problems of~\eqref{eq:PEP} were originally generated from  the vibrational analysis of trail tracks for obtaining information on reducing noise between wheel and rail~\cite{Ipsen04,Hilliges2004}.
Our task in this section is to identify eigenvectors of problem~\eqref{eq:PEP} and then associate these eigenvectors
with the solution of~\eqref{CTS}.
We begin this analyst by studying the eigeninformation of two matrices $A$ and $B$, where $A-\lambda B$ is a regular matrix pencil.
\begin{lemma}\label{LemmAB}
Let $A-\lambda B\in\mathbb{C}^{n\times n}$ be a regular matrix pencil. Assume that matrices $X_i,Y_i\in\mathbb{C}^{n\times n_i}$, $i=1,2$, are full rank and satisfy the following equations
\begin{subequations}\label{eq:abxy}
\begin{eqnarray}
A X_i &=& Y_i R_i,\\
B X_i &=& Y_i S_i,
\end{eqnarray}
\end{subequations}
where $R_i$ and $S_i$, $i=1,2$, are square matrices of size $n_i\times n_i$.
Then
\begin{itemize}
\item [i)] $R_i-\lambda S_i\in\mathbb{C}^{n_i\times n_i}$ are regular matrix pencils for $i=1,2$.

\item [ii)] if $\sigma(R_1-\lambda S_1)\cap \sigma(R_2-\lambda S_2)=\phi$, then the matrix
$\bbb X_1 & X_2\eb\in\mathbb{C}^{n\times (n_1+n_2)}$  is full rank.
\end{itemize}
\end{lemma}
\begin{proof}
Fix $i\in \{1,2\}$. To show that $R_i-\lambda S_i$ is a regular pencil, it suffices to show that if $\lambda$ is an eigenvalue of $R_i-\lambda S_i$, then $\lambda$ is an eigenvalue of $A-\lambda B$.
Let $\mathbf{x} \neq 0$ be an eigenvector of $R_i-\lambda S_i$ corresponding to the eigenvalue $\lambda$, that is,
\begin{align}\label{eq:rs}
\begin{array}{rc}
(R_i - \lambda S_i) \mathbf{x} = 0, & \mbox{if }\lambda < \infty,\\
(S_i  - 0 R_i) \mathbf{x} = 0, & \mbox{if }\lambda = \infty.
\end{array}
\end{align}
We then pre-multiply both sides of~\eqref{eq:rs} by $X_i$ and
obtain
\begin{align*}
\begin{array}{rc}
(A - \lambda B) X_i\mathbf{x} = 0, & \mbox{if }\lambda < \infty,\\
(B  - 0 A) X_i\mathbf{x} = 0, & \mbox{if }\lambda = \infty.
\end{array}
\end{align*}
Since $X_i$ is full rank and $\mathbf{x}\neq 0$, we have $X_i \mathbf{x}\neq 0$. Hence, $\lambda$ is an eigenvalue of $(A, B)$. This proves part i).

Next, by Theorem~\ref{Weierstrass} there exist nonsingular matrices $P_i,\,Q_i$, Jordan block matrices $J_i\in\mathbb{C}^{k_i\times k_i}$, and nilpotent matrices $N_i$, $i=1,2$ (exactly one of  $N_1$ or $N_2$ exists since the regularity of $R_i-\lambda S_i$) such that
\begin{subequations}\label{eq:pq}
\begin{eqnarray}
(P_1R_1Q_1,  P_1S_1Q_1)&=&(\bbb J_1 & \\ & I  \eb, \bbb I & \\ & N_1  \eb),\\
(P_2R_2Q_2, P_2 S_2Q_2)&=& (\bbb J_2 & \\ & I  \eb,  \bbb I & \\ & N_2  \eb).
\end{eqnarray}
\end{subequations}
By~\eqref{eq:abxy} and~\eqref{eq:pq}, it can be seen that
 \begin{align*}
A  X_{i} Q_i  &= Y_i P_i^{-1}  \bbb J_i  & 0\\ 0  & I \eb,\\
B  X_{i} Q_i  &= Y_i P_i^{-1}  \bbb I & 0 \\ 0 & N_i \eb.
\end{align*}
Let $X_iQ_i=\bbb X_{i,1} & X_{i,2} \eb$ and  $Y_iP_i^{-1}=\bbb Y_{i,1} & Y_{i,2} \eb$ be two partitioned matrices with size $n_i \times (k_i+(n_i-k_i)) $ for $i=1,2$. It then follows from direct computation that the matrix pair $A-\lambda B$ satisfies
\begin{align*}
A \bbb X_{1,1} & X_{2,1} \eb &= B  \bbb X_{1,1} & X_{2,1} \eb \bbb J_1  & 0\\ 0  & J_2 \eb,\\
B \bbb X_{1,2} & X_{2,2} \eb &= A  \bbb X_{1,2} & X_{2,2} \eb \bbb N_1 & 0 \\ 0 & N_2 \eb.
\end{align*}

Since $\sigma(R_1-\lambda S_1)\cap \sigma(R_2-\lambda S_2)=\phi$,
we might assume without loss of generality that $N_2$ does not exist, i.e., $(P_2R_2Q_2, P_2 S_2Q_2)=  (J_2,  I )$. Since the condition $\sigma(R_1-\lambda S_1)\cap \sigma(R_2-\lambda S_2)=\phi$ holds, it then follows from Theorem~\ref{lancaster} that the matrix $\bbb X_1 & X_2\eb=\bbb X_{1,1}&X_{2,1}&X_{1,2}\eb$ is full rank.

\end{proof}

Armed with the property given in Lemma~\ref{LemmAB}, we can now attack the problem of determine how the deflating subspace is related to the solution of ~\eqref{CTS}.
\begin {theorem}\label{deflatingthm}
Corresponding to~\eqref{CTS}, let $\mathcal{Z}$ be a matrix defined by~\eqref{eq:Zmat}.
If $\sigma(A^\star-\lambda B)$ is reciprocal free, and $U_i,\,V_i\in\mathbb{C}^{n\times n}$, $i = 1,2$, are matrices satisfying
\begin{subequations}\label{eq:cond1}
\begin{align}
\mathcal{Z}^\star \bbb U_1 \\ V_1 \eb= \bbb U_2 \\ V_2 \eb T_1^\star,\label{d1}\\
\mathcal{Z} \bbb U_1 \\ V_1 \eb= \bbb U_2 \\ V_2 \eb T_2,\label{d2}
\end{align}
\end{subequations}
for some matrices $T_1, T_2 \in \mathbb{C}^{n\times n}$.
%
 Then,
\begin{itemize}
\item [i)] $V_1=U_2=0$  if $T_1^\star-\lambda T_2\sim B^\star-\lambda A$;
\item [ii)]
$V_1$ is nonsingular  if  $T_1^\star-\lambda T_2\sim A^\star-\lambda B$. Moreover, $X=U_1V_1^{-1}=-U_2^{-\star} V_2^\star$ solves the $\star$-Sylvester matrix equation~\eqref{CTS} if $B$ is a nonsingular.
\end{itemize}
\end {theorem}
\begin{proof}

\begin{itemize}
\item[i)]
It follows from \eqref{d1} and \eqref{d2} that
$A^\star V_1 = U_2 T_1^\star$ and 
$B V_1  =U_2 T_2$.
Since $\sigma(A^\star-\lambda B)\cap\sigma(T_1^\star-\lambda T_2)=\phi$, we have $V_1 = U_2 = 0$ by Theorem~\ref{Eric0}.
\item[ii)]
Notice that, since $T_1^\star-\lambda T_2\sim A^\star-\lambda B$,
there exist nonsingular matrices $U$ and $V$ such that
\begin{equation}\label{eq:ccond2}
B^\star U = V T_2^\star \mbox{ and } A U = V T_1.
\end{equation}
Thus, it can be seen that
\begin{subequations}\label{eq:k}
\begin{align}
\mathcal{Z}^\star\bbb U \\ 0 \eb &=\bbb 0 \\ V\eb T_2^\star,\label{space1}\\
\mathcal{Z}\bbb U \\ 0 \eb &=\bbb 0 \\ V\eb T_1. \label{space2}
\end{align}
\end{subequations}

Hence, by~\eqref{eq:cond1} and~\eqref{eq:k}, we have
\begin{subequations}\label{eq:cond2}
\begin{align}
\mathcal{Z}^\star\bbb U & U_1\\ 0& V_1 \eb &=\bbb 0 & U_2\\ V & V_2 \eb
\bbb T_2 & 0\\ 0 &  T_1^\star\eb,\\
\mathcal{Z}\bbb U & U_1\\ 0& V_1 \eb &=\bbb 0 & U_2\\ V & V_2 \eb \bbb T_1^\star & 0\\ 0 &  T_2\eb.
\label{space22}
\end{align}
\end{subequations}

Since $\sigma(\mathcal{Z}^\star - \lambda \mathcal{Z}) = \sigma(A^\star - \lambda B) \cup \sigma(B^\star - \lambda A)$, by Lemma~\ref{LemmAB}, the matrix
$\bbb U & U_1\\ 0& V_1 \eb$ is nonsingular. Thus, $V_1$ is nonsingular.

To show that $X = U_1V_1^{-1}$ is the solution of~\eqref{CTS}, we first claim that $U_2$ is nonsingular.
From~\eqref{eq:ccond2}, we see that $T_2$ is invertible
since matrices $B$, $U$, and $V$ are nonsingular.
From~\eqref{space22}, we know that
$
B V_1=U_2T_2
$.
Since
$B$, $T_2$ and $V_1$ are nonsingular, this implies that $U_2$ is invertible.

Let $\widehat{T}_1=U_2 T_1^\star V_1^{-1}$, $\widehat{T}_2=U_2 T_2 V_1^{-1}$, $X = U_1V_1^{-1}$, and $Y = V_2U_2^{-1}$. It follows from~\eqref{eq:cond2} that
\begin{align*}
A^\star &=  \widehat{T}_1, \quad
B^\star X-C^\star = Y \widehat{T}_1,\\
B &= \widehat{T}_2, \quad
AX-C =Y\widehat{T}_2.
\end{align*}
This implies that
$
A (-Y^\star)+X^\star B =C \mbox{ and }
AX+(-Y)B =C,
$
i.e.,
\begin{align}\label{eqxyy}
A (X+Y^\star) -  (X+Y^\star)^\star B=0.
\end{align}
Since $\sigma(A^\star-\lambda B)$ is reciprocal free, it follows that
$\sigma(A^\star+\lambda B)$ is reciprocal free. Thus, we have $X=-Y^\star$ and $AX + X^\star B = C$ by the uniqueness of the solution of~\eqref{eqxyy}.
\end{itemize}
\end{proof}

Theorem~\ref{deflatingthm} shows that if $A^\star-\lambda B$ is reciprocal free and $B$ is nonsingular, we can solve~\eqref{CTS} by the deflating subspace method. Also, the reciprocal free condition implies that either $A$ or $B$ is nonsingular. Thus, we might assume without loss of generality that $B$ is nonsingular. Otherwise, we can replace $\mathcal{Z}^\star$ of \eqref{d1} with $\mathcal{Z}$ of \eqref{d2} in Theorem~\ref{deflatingthm}.
In the proof of Theorem~\ref{deflatingthm}, we know that if $B$ is nonsingular, then $T_2$ is invertible. We then are able to transform the formulae defined in~\eqref{eq:cond1} into the {palindromic} eigenvalue problem as follows.
\begin{corollary}\label{CTSsol}
%
Suppose that  $\sigma(A^\star - \lambda B)$ is reciprocal free
and the matrix $B$ is nonsingular. If there exists a full rank matrix  $\left[\begin{array}{c}U \\V\end{array}\right]$ such that
\begin{equation}
\mathcal{Z}^\star \left[\begin{array}{c}U \\V\end{array}\right] = \mathcal{Z}  \left[\begin{array}{c}U \\V\end{array}\right]  T,
\end{equation}
for some matrix $T$ with $\sigma(T) = \sigma(A^\star-\lambda B)$, then $V$ is nonsingular and  $X= UV^{-1}$
is the unique solution of the $\star$-Sylvester matrix equation~\eqref{CTS}.
\end{corollary}
From Corollary~\ref{CTSsol}, it can be seen that the solution of~\eqref{CTS} can be obtained by solving
the palindromic eigenvalue problem~\eqref{eq:PEP}.
We refer the reader to~\cite{Chu2008,Chu2010, Hilliges2004} for more details. 
It should also be noted that the assumption of the existence of a full rank matrix $\bbb U \\ V\eb$ is always true, since the eigenvalues of $\mathcal{Z}^\star  -\lambda\mathcal{Z}$  are composed of $\sigma(A^\star - \lambda B) \cup \sigma(B^\star - \lambda A)$.

\section{An efficient iterative method}
In this section, we want to discuss how to apply  the palindromic doubling algorithm (abbreviated as PDA)  \cite{Li2011} to find the stabilizing solution of ~\eqref{CTS}. The stabilizing solution $X$ of algebraic Riccati equations has been an extremely active area of the design of feedback controller \cite{Lancaster1995}. Our interest in the stabilizing solution $X$ of ~\eqref{CTS} originates from the solution of the $\star$-Riccati matrix equation
\begin{align}\label{starRic}
XAX^\star+XB+CX^\star+D=0
\end{align}
from an application related to the palindromic eigenvalue problem \cite{Hilliges2004,Chu2010}. Finding the solution of the $\star$-Riccati matrix equation is a difficult treatment. The application of Newton's method is a reasonable possibility and leads to the iterative process
\begin{eqnarray}\label{NW}
(C+X_kA)X_{k+1}^\star+X_{k+1}(B+AX_k^\star)=X_k A X_k^\star-D,
\end{eqnarray}
which is $\star$-Sylvester matrix equation with respect to $X_{k+1}$ for nonnegative integer $k$. To guarantee the convergence of the Newton's method \eqref{NW}, under some mild assumptions on coefficient matrices of $\star$-Riccati matrix equation~\eqref{NW}, one can choose the initial value $X_0$ such that the spectrum set of $\sigma(B+AX_k A-\lambda (C^\star+A^\star X_k^\star))$ lies in unit circle \cite{Chiang2012} for nonnegative integer $k$ so that the Newton's method will quadratically converge to the stabilizing solution.  We focus on how efficiently solve the $\star$-Sylvester matrix equation in this paper. According to above discussion, the definition of stabilizing solution $X$ of ~\eqref{CTS} is stated as follows,
\begin{Definition}\label{stab}
If $\sigma(A^\star-\lambda B)$ lies in the unit circle in ~\eqref{CTS}, then the solution $X$ (if exist) of ~\eqref{CTS} is called a stabilizing solution.
\end{Definition}
For a matrix pencil $A^\star-\lambda B$, if $\sigma(A^\star-\lambda B)$ lies in the unit circle, it is clear that $\sigma(A^\star-\lambda B)$ is $\star$-reciprocal free. It follows that the stabilizing solution always exists and is unique.
Our approach  in the next subsection is to develop an efficient numerical algorithm for finding the stabilizing solution of \eqref{CTS}.
To this end, we start with a review of the so-called palindromic doubling algorithm. This method has been applied to obtain the stabilizing solutions of the generalized continuous-time (and the generalized discrete-time) algebraic Riccati equations~\cite{Li2011}.

\subsection{Palindromic Doubling Algorithm}
Given a matrix pencil $\mathcal{A}-\lambda \mathcal{B}$ and assume $-1\not\in \sigma (\mathcal{A}-\lambda\mathcal{B})$, since $ \mathcal{B} (\mathcal{A}+ \mathcal{B})^{-1} (\mathcal{A}+ \mathcal{B}) =  \mathcal{B} =
(\mathcal{A}+ \mathcal{B}) (\mathcal{A}+ \mathcal{B})^{-1} \mathcal{B}$, it is easy to see that
\begin{eqnarray}\label{swap}
\mathcal{B}(\mathcal{A}+\mathcal{B})^{-1} \mathcal{A}
&=& \mathcal{B} -\mathcal{B}(\mathcal{A}+\mathcal{B})^{-1} \mathcal{B}
= \mathcal{A}(\mathcal{A}+\mathcal{B})^{-1} \mathcal{B}\label{ZZstar}.
\end{eqnarray}
We now consider the \emph{doubling transformation} $\mathcal{A}-\lambda \mathcal{B}\rightarrow \widehat{\mathcal{A}}-\lambda \widehat{\mathcal{B}}$ by
\begin{subequations}\label{doublingtrans}
\begin{align}
\widehat{\mathcal{A}} = \mathcal{A}(\mathcal{A}+\mathcal{B})^{-1}\mathcal{A},\\
\widehat{\mathcal{B}} = \mathcal{B}(\mathcal{A}+\mathcal{B})^{-1}\mathcal{B}.
\end{align}
\end{subequations}
The following theorem is to show that such transformation will keep the eigenspace unchanged and double the original eigenvalues.

\begin{theorem}\label{doublingtransthm}
The matrix pair $\widehat{\mathcal{A}}-\lambda \widehat{\mathcal{B}}$ has the doubling
property, i.e., if
\begin{align}\label{CAB}
\mathcal{A} \bbb U\\V\eb=\mathcal{B} \bbb U\\V\eb {T},
\end{align}
where $U,V\in\mathbb{C}^{2n\times n}$ and
$ T\in\mathbb{C}^{n\times n}$, then
\begin{align}\label{CAB2}
\widehat{\mathcal{A}} \bbb U\\V\eb=\widehat{\mathcal{B}} \bbb U\\V\eb {T}^2.
\end{align}
\end{theorem}
\begin{proof}
Pre-multiplying the both sides of \eqref{CAB} by
$\mathcal{A}(\mathcal{A}+\mathcal{B})^{-1}$ and applying \eqref{swap} and \eqref{doublingtrans}, we obtain \eqref{CAB2}, which completes the proof.
%
%
%
%

\end{proof}

To see how the doubling transformation can be applied to obtain the stabilizing solution of \eqref{CTS}, write the matrix $Z$ in~\eqref{eq:Zmat} as
\begin{subequations}\label{PDA}
\begin{align}\label{DA1}
\mathcal{Z}=\mathcal{H}+\mathcal{K},
\end{align}
where
\begin{equation}\label{DA2}
\mathcal{H}=\frac{1}{2}(\mathcal{Z}^\star+\mathcal{Z})=\mathcal{H}^\star,\quad
\mathcal{K}=\frac{1}{2}(-\mathcal{Z}^\star+\mathcal{Z})=-\mathcal{K}^\star.
\end{equation}
\end{subequations}
Note that matrices $\mathcal{H}$ and $\mathcal{K}$ are the Hermitian(symmetric) part and skew-Hermitian(skew-symmetric) part of $ \mathcal{Z}$ with $\star=H(\top)$, respectively.
Since $\mathcal{Z}$ is an upper anti-triangular block matrix, we proceed to express $\mathcal{K}$ as
\begin{align}\label{matK}
\mathcal{K} =\bbb 0 & K_{12}\\ -(K_{12})^\star&-K_{22}\eb,
\end{align}
and define an initial matrix $\mathcal{H}_0$ for the PDA as
\begin{align}\label{matH}
\mathcal{H}_0=\bbb 0 & H_{12}^{(0)}\\ (H_{12}^{(0)})^\star& -H_{22}^{(0)}\eb,
\end{align}
where $H_{12}^{(0)}=\frac{A^\star+B}{2}$ and
$H_{22}^{(0)}=\frac{C^\star+C}{2}$,
$K_{12}=\frac{-A^\star+B}{2}$ and
$K_{22}=\frac{-C^\star+C}{2}$. Based on the above notation, we then generalize the PDA technique given in~\cite{Li2011} for obtaining the stabilizing solution of~\eqref{CTS} as follows.

\vspace*{0.3cm}
\noindent{\textbf{Algorithm 1: the PDA}} \\
Given matrices $\mathcal{K}$ and $\mathcal{H}_0$ as defined by~\eqref{matK} and~\eqref{matH},\\
\textbf{for} $k=0,1,\ldots$, \textbf{compute until convergence} \\
\begin{subequations}\label{Ntmalg}
\begin{eqnarray}
T_k&=&(H_{12}^{(k)})^{-1} K_{12},\\
H_{12}^{(k+1)}&=& \frac{1}{2}(H_{12}^{(k)}+K_{12}T_k),\\
H_{22}^{(k+1)}&=&\frac{1}{2}(H_{22}^{(k)}+T_k^\star H_{22}^{(k)}T_k
+K_{22}T_k-T_k^\star K_{22}).
\end{eqnarray}
\end{subequations}
\textbf{end for}\\
\textbf{End of algorithm} \vspace*{0.3cm}\\

Of particular interest is that the stabilizing solution $X$ of \eqref{CTS}
can be represented by
\begin{equation*}
X=\lim\limits_{k\rightarrow\infty} X_k,
\end{equation*}
where
\begin{align}\label{eq:xk}
X_k=(H_{12}^{(k)}+K_{12})^{-\star}(H_{22}^{(k)}+K_{22})^\star.
\end{align}
We prove this result in the following theorem.
\begin{theorem}\label{PDA Solvent}
Suppose that $\sigma(A-\lambda B^\star)$ lies in the unit circle and let $X$ be the stabilizing solution of~\eqref{CTS}. Then, all iterations given in Algorithm~1 are well-defined,  and the sequences $\{Z_{12}^{(k)}, Z_{21}^{(k)}\}$ (See the below definition in (33)) and $\{X_k\}$ in~\eqref{eq:xk} satisfy
%
\begin{eqnarray*}
Z_{12}^{(k)} & \rightarrow& -A^\star+B, \mbox{ quadratically as } k\rightarrow\infty,\\
Z_{21}^{(k)} &\rightarrow& 0, \mbox{ quadratically as } k\rightarrow\infty,\\
X_k &\rightarrow& X, \mbox{quadratically as }k\rightarrow\infty,
\end{eqnarray*}
with convergence rate $\rho$ define by $\rho = \max\limits_{\tau\in\sigma(A-\lambda B^\star)} |\tau|<1$.
\end{theorem}
\begin{proof}
Let $\{\mathcal{H}_k\}$ be a sequence given by
\begin{align*}
\mathcal{H}_k=\bbb 0 & H_{12}^{(k)}\\ (H_{12}^{(k)})^\star& -H_{22}^{(k)}\eb.
\end{align*}
Following from a direct computation of the inverse of $\mathcal{H}_k$ (if exist) and the fact that
\begin{equation*}
 \mathcal{H}_k^{-1} = \bbb (H_{12}^{(k)})^{-\star} H_{22}^{(k)} (H_{12}^{(k)})^{-1} & (H_{12}^{(k)})^{-\star} \\ (H_{12}^{(k)})^{-1} & 0 \eb,
\end{equation*}
we obtain
\begin{align}\label{HK}
\mathcal{H}_{k+1}=\frac{1}{2}(\mathcal{H}_k+\mathcal{K}\mathcal{H}_k^{-1}\mathcal{K}).
\end{align}
Now consider a sequence of iterations given by
\begin{align}\label{eq:Zk}
\mathcal{Z}_k=\mathcal{H}_k+\mathcal{K}\equiv \bbb 0 & Z_{12}^{(k)}\\ Z_{21}^{(k)} & -Z_{22}^{(k)}\eb,
\end{align}
for $k=0,1,\cdots$. It follows from~\eqref{HK} that
\begin{align*}
\mathcal{Z}_{k+1}&=\mathcal{H}_{k+1}+\mathcal{K}=\frac{1}{2}(\mathcal{H}_k+\mathcal{K}+\mathcal{K}\mathcal{H}_k^{-1}\mathcal{K}
+\mathcal{K})\\
&=\frac{1}{2}(\mathcal{H}_k+\mathcal{K})\mathcal{H}_k^{-1}(\mathcal{H}_k+\mathcal{K})=\mathcal{Z}_k(\mathcal{Z}_k+\mathcal{Z}_k^\star)^{-1}\mathcal{Z}_k,
\end{align*}
which is exactly the doubling transformation in \eqref{doublingtrans} with $\mathcal{A}\rightarrow\mathcal{Z}_k^\star$ and $\mathcal{B}\rightarrow\mathcal{Z}_k$. Furthermore, the anti-diagonal block matrix of $\mathcal{Z}_k$ satisfies
\begin{align*}
Z_{12}^{(k+1)}&=Z_{12}^{(k)}(Z_{12}^{(k)}+(Z_{21}^{(k)})^\star)^{-1}Z_{12}^{(k)},\\
Z_{21}^{(k+1)}&=Z_{21}^{(k)}(Z_{21}^{(k)}+(Z_{12}^{(k)})^\star)^{-1}Z_{21}^{(k)},
\end{align*}
i.e., the matrix pencil $(Z_{21}^{(k+1)})^\star-\lambda Z_{12}^{(k+1)}$ is a doubling transformation of the
matrix pencil $(Z_{21}^{(k)})^\star-\lambda Z_{12}^{(k)}$. Thus, $\sigma((Z_{21}^{(k)})^\star-\lambda Z_{12}^{(k)})$ lies in the unit circle for each $k$, since $(Z_{21}^{(0)})^\star-\lambda Z_{12}^{(0)}=A^\star-\lambda B$. We thus conclude that $Z_{12}^{(k)}$ is invertible for all $k$ by the regularity of $A^\star-\lambda B$. On the other hand,
according to Theorem~\ref{deflatingthm}, the initial value $\mathcal{Z}_0=\mathcal{Z}$ gives rise to the fact that
\begin{equation}\label{init}
\mathcal{Z}_0^\star \left[\begin{array}{c} U \\ V\end{array}\right] = \mathcal{Z}_0  \left[\begin{array}{c}U \\V\end{array}\right]  T,
\end{equation}
for some matrix $T$ with $\sigma(T) = \sigma(A^\star-\lambda B)$ and a nonsingular matrix $V$. This implies that the intersection of the spectrum set $\mathcal{Z}_0^\star-\lambda \mathcal{Z}_0$ and the unit circle is empty. It follows that $-1\not\in\sigma(\mathcal{Z}_k^\star-\lambda \mathcal{Z}_k)$ for each $k$ from Theorem~\ref{doublingtransthm}. Thus, $\mathcal{Z}_k^\star+\mathcal{Z}_k$ is invertible and hence $\mathcal{H}_k=\frac{1}{2}(\mathcal{Z}_k^\star+\mathcal{Z}_k)$ is also a nonsingular matrix for all $k$. That is, all iterations in Algorithm~1 are well defined.
It follows from Theorem~\ref{doublingtransthm} that
\begin{equation}\label{doubspac}
\mathcal{Z}_k^\star \left[\begin{array}{c} U \\ V\end{array}\right] = \mathcal{Z}_k  \left[\begin{array}{c}U \\V\end{array}\right]  T^{2^k},
\end{equation}
that is,
\begin{subequations}
\begin{align}
({Z}_{21}^{(k)})^{\star} V &={Z}_{12}^{(k)} V T^{2^k},\label{eq:sub1}\\
({Z}_{12}^{(k)})^{\star} U-({Z}_{22}^{(k)})^{\star} V &= (({Z}_{21}^{(k)})^{\star} U-({Z}_{22}^{(k)})^{\star} V )T^{2^k}.\label{eq:sub2}
\end{align}
\end{subequations}
By equalities~\eqref{eq:Zk} and~\eqref{eq:sub1}, we have
\begin{align*}
(H_{12}^{(k)}-K_{12}) V &=(H_{12}^{(k)}+K_{12}) V T^{2^k},
\end{align*}
that is,
\begin{align*}
H_{12}^{(k)} - K_{12} =2 K_{12} V T^{2^k}(I-T^{2^k})^{-1} V^{-1}.
\end{align*}
This implies that $\lim\limits_{k\rightarrow\infty}H_{12}^{(k)}=K_{12} $. Thus we conclude that
\begin{align*}
\lim\limits_{k\rightarrow\infty}Z_{21}^{(k)}=\lim\limits_{k\rightarrow\infty}(H_{12}^{(k)}-K_{12}^{(0)})^\star=0,\,\,
\lim\limits_{k\rightarrow\infty}Z_{12}^{(k)}=\lim\limits_{k\rightarrow\infty}(H_{12}^{(k)}+K_{12}^{(0)})=-A^\star+B.
\end{align*}

Note that $Z_{21}^{(k)}$ and $Z_{12}^{(k)}$  converge quadratically to $0$ and $-A^\star+B$ with rate $\rho(T)=\rho$, respectively. On the other hand,  it follows from~\eqref{eq:sub2} and~\eqref{eq:Zk} that
\begin{align*}
(H_{12}^{(k)})^\star U-H_{22}^{(k)} V=-(K_{12}^\star U+K_{22}V)(I+T^{2^k})(I-T^{2^k})^{-1}.
\end{align*}
Thus, the right hand side of \eqref{eq:sub2} can be expressed as
\begin{align}\label{eq:sub3}
&({Z}_{21}^{(k)} U-{Z}_{22}^{(k)} V )T^{2^k}=((H_{12}^{(k)})^\star U-H_{22}^{(k)} V)T^{2^k}-(K_{12}^\star U+K_{22}V)T^{2^k}\nonumber\\
&=-(K_{12}^\star U+K_{22}V)(I+(I+T^{2^k})(I-T^{2^k})^{-1})T^{2^k}.
\end{align}

Let $X = UV^{-1}$. It follows from Theorem~\ref{doublingtransthm} that $X$ is a solution of~\eqref{doubspac}.
Pre-multiplying and post-multiplying \eqref{eq:sub2} by $({Z}_{12}^{(k)})^{-\star}$ and $V^{-1}$, respectively, we have
%
%
\begin{align*}
X-X_k=({Z}_{12}^{(k)})^{-\star}({Z}_{21}^{(k)} U-{Z}_{22}^{(k)} V )T^{2^k}V^{-1},
\end{align*}
where
\begin{equation*}
X_k= (H_{12}^{(k)}+K_{12})^{-\star}(H_{22}^{(k)}+K_{22})^\star = ({Z}_{12}^{(k)})^{-\star}({Z}_{22}^{(k)})^{\star}.
\end{equation*}
From~\eqref{eq:sub3} we have $\limsup\limits_{k\rightarrow\infty}\sqrt[2^k]{\| X-X_k\|}\leq \max\limits_{\tau\in\sigma(A-\lambda B^\star)} |\tau|=\rho$, which completes the proof.
\end{proof}

%

From Theorem~\ref{PDA Solvent}, it is natural to modify the definition $X_k$ in~\eqref{eq:xk} by solving the linear system
\begin{align*}
(-A+B^\star) X_k=Z_{22}^{(k)},\,\, k=1,2,\cdots
\end{align*}
Besides, the capacity of Algorithm~1 for solving~\eqref{CTS} with eigenvalues of $\sigma(A^\star -\lambda B)$ lying on the unit circle
is something worthy of our discussion.
Note that  for the case $\star={H}$, there does not exist any unimodular eigenvalue lying on the unit circle under the condition of the uniquely solvable solution of \eqref{CTS}. However, from Theorem~\ref{solvable} we know that if $1$ is a simple eigenvalue of $\sigma(A^\top -\lambda B)$ and
$\sigma(A^\top-\lambda B)\backslash \{1\}$ is reciprocal free, then the equation~\eqref{CTS} with $\star= \top$ is still solvable.
In this critical case, we call the solution as the \emph{almost} stabilizing solution~\cite{Lancaster1995}.
It is interesting to know that Algorithm~1 can also be applied to solve the critical case of~\eqref{CTS} with no difficulty. For the convergence analysis of this algorithm, we first consider the following analysis of the eigenstructure of the unimodular eigenvalues of $\mathcal{Z}^\top- \lambda \mathcal{Z}$ (i.e., $\lambda=1$).

%
%
%
\begin{lemma}\label{Lem}
Suppose that $1\in\sigma(A^\top-\lambda B)$
and $1$ is simple. Then $\mbox{nullity}(\mathcal{Z}^\top-\mathcal{Z})=2$.
\end{lemma}
\begin{proof}
Considering two orthogonal matrices $Q_1$ and $Z_1$, let $Z_1 A Q_1^\top := \widehat{A}= [\hat{a}_{ij}]$
and $ Q_1 B Z_1^\top :=  \widehat{B} = [\hat{b}_{ij}]$  be the QZ or generalized Schur decomposition of  $A^\top$ and $B$ so that $\widehat{A}^\top$ and $\widehat{B}$ are upper-triangular. Corresponding to this decomposition, the matrix pencil $Z^\top - Z$ can be expressed as
\begin{equation}
\mathcal{Z}^\top - \mathcal{Z} =  \bbb Q_1^\top & 0\\ 0 &Z_1^\top \eb
 \left[\begin{array}{cc}
 0 & \widehat{A}^\top-\widehat{B}
 \\
 \widehat{B}^\top-\widehat{A} & \widehat{C}-\widehat{C}^\top
 \end{array}\right]
 \bbb Q_1 & 0\\ 0 & Z_1 \eb
\end{equation}
with $\widehat{C} := Z_1 C  Z_1^\top$.
Thus, in order to discuss
the dimension of the null space  of the matrix $\mathcal{Z}^\top - \mathcal{Z}$, we can assume without loss of generality that  $A^\top$ and $B$ are upper triangular matrices and $a_{11}=b_{11}$.
It follows that there exist
a nonzero vector $\mathbf{x}_0$ and an unit vector  $\mathbf{e}_1$  such that  vectors $\mathbf{x}_0$ and $\mathbf{e}_1$ are in the null space of matrices $B^\top-A$ and $A^\top-B$, respectively.
Since the first row of $B^\top-A$ is a zero row vector and the other rows of $B^\top-A$ are linearly independent, there exists a vector $\mathbf{x}_1$ satisfying
$(B^\top-A)\mathbf{x}_1=(C^\top-C)\mathbf{e}_1$.

Recall also from~\eqref{eq:PEP} that
\begin{align*}
\mathcal{Z}^\top-\mathcal{Z}=\bbb 0 & A^\top-B\\B^\top-A & C-C^\top\eb.
\end{align*}
We see that  $\mathbf{v}_1=\bbb \mathbf{x}_1 \\ \mathbf{e}_1\eb$ and $\mathbf{v}_2=\bbb \mathbf{x}_0\\ \mathbf{0}\eb$
 are two linearly independent vectors in the null space of $\mathcal{Z}^\top-\mathcal{Z}$, that is, $\mbox{nullity}(\mathcal{Z}^\top-\mathcal{Z})=2$.
\end{proof}

Lemma~\ref{Lem} tells that the eigenvalue $``1''$ of $\mathcal{Z}^\top-\lambda\mathcal{Z}$ has partial multiplicity one (two $1\times 1$ Jordan blocks) if
the eigenvalue $``1''$ of $A^\top-\lambda B$ is simple.
Based on Lemma~\ref{Lem}, it can be shown that Algorithm~1 can be applied to solve this problem with a linear rate of convergence. Because the analysis of the convergence property  of Algorithm 1 is an analogous result for the palindromic generalized eigenvalue problem in~\cite{Li2011}[Theorem 3.1], we omit it here.
%
\begin{theorem}\label{DAconvthm}
In the critical case, all sequences generated in Algorithm~1
for finding the almost stabilizing solution $X$ of~\eqref{CTS}
are well defined. Moreover,
\begin{eqnarray*}
Z_{12}^{(k)}&\rightarrow& -A^T+B,\,\mbox{linearly}\,\,
\mbox{as } k\rightarrow\infty, \\
Z_{21}^{(k)} &\rightarrow& 0,\,\mbox{linearly}\,\,
\mbox{as } k\rightarrow\infty, \\
X_k &\rightarrow& X,\,\mbox{linearly}\,\,
\mbox{as }  k\rightarrow\infty,
\end{eqnarray*}
with convergence rate at least $1/2$.
\end{theorem}
%

\subsection{A numerical example}
In this subsection, we use a numerical example to illustrate the efficiency of Algorithm~1 with the assumption of Theorem~\ref{PDA Solvent} and to demonstrate the numerical behavior of Algorithm~1 in the critical case. All computations were performed in MATLAB/version 2011 on a PC with an Intel Core i7-4770 3.40 GHZ processor and 32 GB main memory. For evaluating the performance, we define the error (ERR), relative error (RERR) and relative normalized residual (RES) as follows
\begin{eqnarray*}
\mbox{ERR} &\equiv& \| Z_{21}^{(k)}\|_F,\,\,\mbox{RERR}\equiv \frac{\| X_k-X\|_F}{\|X\|_F}, \\
\mbox{RES}&\equiv& \frac{\| AX_k+X_k^\star B-C\|_F}{\|A\|_F\|X_k\|_F+\|B\|_F\|X_k\|_F+\|C\|_F},
\end{eqnarray*}
respectively, where $X$ is the (almost) stabilizing solution of \eqref{CTS}. All iterations are terminated whenever the errors or relative errors or the relative normalized residual residuals are less than $n^2\mathbf{u}$, where $\mathbf{u}=2^{-52}\cong 2.22e-16$ is the machine zero.
\begin{example}~\label{ex1}
Let  $\widehat{A}^\top, \widehat{B} \in \mathbb{R}^{n \times n}$ be two real lower-triangular matrices with given diagonal elements (specified by $a,b \in \mathbb{R}^n$) and random strictly lower-triangular elements. They are then reshuffled by the orthogonal matrices $Q,Z \in \mathbb{R}^{n \times n}$ to form $(A,B) = (Q \widehat{A} Z, Q \widehat{B} Z)$, that is, in \texttt{MATLAB} commands, we define
\begin{eqnarray*}
\widehat{A} &=&triu(randn(n),-1)+diag(a), \,\,
\widehat{B}=tril(randn(n),-1)+diag(b), \\
C&=&AX+X^\top B,
\end{eqnarray*}
where
\[a=[temp_1.*temp_2,1-\epsilon] \mbox{ and } b=[temp_1;1]\]
with
\begin{align*}
temp_1=rand(n-1,1),\,\,temp_2=rand(n-1,1).
\end{align*}
where $X=randn(n)$ is given as the (almost) stabilizing solution of \eqref{CTS} and $0\leq \epsilon <1$. The setup of $a$ and $b$ guarantee that $\sigma(A^T-\lambda B)$ lies within the unit circle if $\epsilon\neq 0$, and ``1'' will be a simple eigenvalue of $\sigma(A^T-\lambda B)$ if $\epsilon=0$. Let $n=10$, Table~\ref{table1} contains the iteration numbers (refer as ``ITs''), ERR, RERR and RES for various numbers of $\epsilon$. We see that Algorithm~1 quadratically converges to the stabilizing solution $X$ if $\epsilon=1e-1,1e-2$, and all measurements for errors almost have the same accuracy for every tests. Since the limit of the sequence $\{Z_{12}^k\}$ is
$-A^\top + B$, the condition of $\{Z_{12}^k\}$ becomes ill-conditioning as $\epsilon$ approaches zero.
Note that the sequence $\{X_k\}$ generated by Algorithm~1 converges well to the stabilizing solution before the sequence
 $\{Z_{12}^k\}$ tends to a singular matrix. However, the number of iterations required in the computation dramatically increases, once $\epsilon$ approaches zero and finally loses the property of quadratic convergence. Also, it can be observed that forward errors of the almost stable eigenvalues are approximately equal to $\sqrt{\mathbf{u}}$. Similar phenomena can also be seen in~\cite{Li2011} for the study of the palindromic generalized eigenvalue problem.

%
%
%
%
%

\begin{table}[htbn]
\begin{center}
\begin{tabular}[c]{c|cccc} & &  \\[-2.4ex]
 $\epsilon$    & \multicolumn{1}{c}{ITs}  &
       \multicolumn{1}{c}{ERR} &
     \multicolumn{1}{c}{RERR} & \multicolumn{1}{c}{RES}
                                      \\
     \hline  & & & \\[-2.4ex]
      1e-1  & 5.00e+00 &    9.2527e-17&3.9015e-18 & 8.1211e-16 \\
      1e-2  & 7.00e+00 &   9.8965e-15& 5.8186e-15 & 8.0061e-16 \\
      1e-4  & 1.60e+01 &   1.7175e-13& 1.1288e-13  &  7.2224e-14 \\
      1e-8  & 2.40e+01 &   2.8457e-10& 1.3213e-10 & 1.0476e-10 \\
      0  & 3.80e+01 &   6.6106e-8&  3.9109e-8 & 3.4419e-8 \\
\end{tabular}
\caption{Results for Example~\ref{ex1}} \label{table1}
\end{center}
\end{table}

\end{example}
\section{Conclusion}
One common procedure to solve the Sylvester equations is by means of the Schur decomposition. In this paper, we present the invariant subspace method and, more generally, the deflating subspace method to solve the Sylvester equations. Our methods are based on the analysis of the eigeninformation of two square matrices defined in Section 3 and 4. We carry out a thorough discussion to address the various eigeninformation encountered in the subspace methods. These ideas can then be implemented into a doubling algorithm for solving the (almost) stabilizing solution of \eqref{CTS} in Section~5. On the other hand, it follows from Theorem~\ref{solvable} that~\eqref{CTS} still has an unique solution, even if some eigenvalues of $A-\lambda B^\star$ lie beyond the unit circle. How to apply Algorithm~1 for solving this problem is under investigation and will be reported elsewhere.

\section*{Acknowledgement}
The authors wish to thank editor and two anonymous referees for many interesting and valuable suggestions on the manuscript.


\end{document}